\documentclass[12pt]{article}

\usepackage[T1]{fontenc}
\usepackage[utf8]{inputenc}
\usepackage[a4paper,margin=1in]{geometry}
\usepackage{amsmath,amssymb,amsthm,mathtools}
\usepackage{enumitem}
\usepackage{booktabs}
\usepackage{microtype}
\usepackage{graphicx}
\usepackage{xcolor}
\usepackage{tikz}
\usetikzlibrary{calc,backgrounds,positioning,fit,arrows.meta}
\usepackage[colorlinks=true,linkcolor=blue,citecolor=blue,urlcolor=blue]{hyperref}

\newtheorem{theorem}{Theorem}
\newtheorem{lemma}{Lemma}
\newtheorem{proposition}{Proposition}

\newtheorem{question}{Question}

\newcommand{\F}{\mathbb F}

\usetikzlibrary{calc,backgrounds,fit}
\definecolor{portzero}{RGB}{200,30,30}
\definecolor{portone}{RGB}{0,114,178}
\definecolor{porttwo}{RGB}{204, 153, 0} 
\definecolor{layergray}{RGB}{245,245,245}
\tikzset{
  every node/.style={font=\small},
  internal/.style={black,line width=0.65pt},
  labelnode/.style={font=\scriptsize,fill=white,inner sep=1.0pt,text=black},
  smalllabel/.style={font=\scriptsize,fill=white,inner sep=0.8pt,text=black},
  v0/.style={circle,draw=portzero,fill=white,line width=0.55pt,inner sep=1.0pt,minimum size=4.4pt},
  v1/.style={circle,draw=portone,fill=white,line width=0.55pt,inner sep=1.0pt,minimum size=4.4pt},
  v2/.style={circle,draw=porttwo,fill=white,line width=0.55pt,inner sep=1.0pt,minimum size=4.4pt},
  p0edge/.style={portzero,line width=0.58pt,opacity=.70,shorten <=2pt,shorten >=2pt},
  p1edge/.style={portone,line width=0.58pt,dash pattern=on 2.3pt off 1.2pt,opacity=.70,shorten <=2pt,shorten >=2pt},
  p2edge/.style={porttwo,line width=0.62pt,densely dotted,opacity=.82,shorten <=2pt,shorten >=2pt},
  legendline/.style={line width=0.8pt}
}

\title{On $k$-connected vertex-pancyclic graphs without pancyclic edges}
\author{Leyou Xu\footnote{Email: leyouxu@m.scnu.edu.cn}, Bo Zhou\footnote{Email: zhoubo@m.scnu.edu.cn} \\
School of Mathematical Sciences, South China Normal University\\
Guangzhou 510631, P.R. China}
\date{}

\begin{document}
\maketitle

\begin{abstract}
An edge of a  graph of order $n$  is  pancyclic if it lies in a cycle of every length $3,\ldots,n$. A graph of order $n$  is vertex-pancyclic if every vertex lies in a cycle of every length $3,\ldots,n$. 
Recently, Li and Zhan proved that  every $2$-connected $[4,2]$-graph 
of order at least seven   contains a pancyclic edge. Zhan asked whether there exists a positive integer $k$ such that every $k$-connected vertex-pancyclic graph contains a pancyclic edge. We answer this question by showing that 
for every positive integer $k$, there is a $k$-connected vertex-pancyclic graph containing no pancyclic edge.
\end{abstract}


\section{Introduction}

All graphs in this note are finite and simple. 
Pancyclicity is one of the standard ways of strengthening Hamiltonicity. A Hamiltonian graph contains a spanning cycle; a pancyclic graph contains cycles of every possible length. Since Bondy's foundational theorem on pancyclic graphs \cite{Bondy1971}, a natural line of research has been to determine 
when a condition that forces Hamiltonicity also forces pancyclicity; see, for example, the survey of Gould \cite{Gould2003} and the pancyclicity results of Bondy and Ingleton \cite{BondyIngleton1976} and Schmeichel and Hakimi \cite{SchmeichelHakimi1974}. This perspective is often summarized by Bondy's meta-conjectural principle that nontrivial Hamiltonicity hypotheses tend to imply pancyclicity, apart from a few natural exceptional families.


Let $G$ be a graph of order $n$, where $n\ge 3$. The graph $G$  is \emph{pancyclic} if it contains an $m$-cycle (a cycle of length $m$) for every $m\in\{3,\ldots,n\}$. A vertex $v\in V(G)$ is called pancyclic if it lies in an $m$-cycle for every such $m$, and $G$ is \emph{vertex-pancyclic} if every vertex of $G$ is pancyclic. Similarly, an edge $e\in E(G)$ is called a \emph{pancyclic edge} if, for every $m\in\{3,\ldots,n\}$, the edge $e$ lies in an $m$-cycle, and $G$ is  \emph{edge-pancyclic} if
every edge of $G$ is pancyclic.
Evidently,  an edge-pancyclic graph is vertex-pancyclic, and a  vertex-pancyclic graph is pancyclic. The reverse implications fail in general, and a substantial body of work has studied degree, closure, extremal, and structural hypotheses under which such local cycle properties (vertex-pancyclicity and edge-pancyclicity) follow; see \cite{Broersma1997,CreamGouldHirohata2018,Hendry1990,Hobbs1976,Randerath2002,ZhangHolton1993}.

A graph $G$ is $k$-connected if $|V(G)|>k$ and deleting any set of at most $k-1$ vertices leaves the graph connected. The minimum  such $k$ is the connectivity of $G$, written as $\kappa(G)$. This notion appeared in classical Hamiltonicity theory, for example, the Chv\'{a}tal-Erd\H{o}s theorem states that for any graph $G$ at least three vertices, if it is $\alpha(G)$-connected, then it  is Hamiltonian \cite{ChvatalErdos1972}, where $\alpha (G)$ is the independence number of $G$. 


A graph is called  an $[s,t]$-graph if every induced subgraph on $s$ vertices has at least $t$ edges. A graph $G$ with $k\ge \alpha(G)$ is is just  a $[k+1,1]$-graph. Recently, Li and Zhan \cite{LiZhan2026} proved that every $2$-connected $[4,2]$-graph of order at least seven contains a pancyclic edge. 
They also found a vertex-pancyclic graph of order~$12$ with no pancyclic edge. 
So vertex-pancyclicity alone is insufficient to guarantee a pancyclic edge. 
After thinking of the problem to characterize the vertex-pancyclic graphs that contain no pancyclic edge, Zhan posed the following question in \cite{LiZhan2026}.

\begin{question} \label{ques:zhan}
Does there exist a positive integer $k$ such that every $k$-connected vertex-pancyclic graph contains a pancyclic edge?
\end{question}


In this note, we provide an answer to this question by proving the theorem below.

\begin{theorem}\label{thm:main}
For every positive integer $k$, there exists a $k$-connected vertex-pancyclic graph that contains no pancyclic edge.
\end{theorem}

More precisely, for every positive integer $k$,  we construct a  graph $G_k$ of order $24k$ such that $G_k$ is $k$-connected (with connectivity $2k+2$) and vertex-pancyclic, but no edge of $G_k$ is pancyclic.

By Theorem \ref{thm:main}, there are vertex-pancyclic graphs with arbitrarily high connectivity yet still no pancyclic edge, implying that  a characterization of the vertex-pancyclic graphs that contain no pancyclic edge must involve other structural invariants.

The construction  is designed to answer Question~\ref{ques:zhan}. It does not contradict positive results for more restrictive graph classes.

The rest of the article is organized as follows. 
Section~\ref{sec:construction} defines the graphs $G_k$ and gives a compressed visual schematic of a port and an example. 
Section~\ref{sec:proof-main} proves Theorem~\ref{thm:main} by showing that $G_k$ is such a graph. 


\section{The construction}\label{sec:construction}


Let $A$ be a subset of an abelian group. We say that $A$ is \emph{sum-free} if $a+b\notin A$ for all $a,b\in A$.

Let $Q=\F_2^3$.
We write the elements of $Q$ as binary triples, and addition in $Q$ is coordinatewise modulo $2$. Put
\[
P_0=\{100,010\},\qquad
P_1=\{001,110,101\},\qquad
P_2=\{011,111\}.
\]
Then $P_0,P_1,P_2$ form a partition of $Q\setminus\{000\}$.

\begin{lemma}\label{lem:sumfree}
Each of $P_0,P_1,P_2$  is sum-free.
\end{lemma}

\begin{proof}  For each $i=0,1,2$, if $a\in P_i$, then 
$a+a=000\notin P_i$. It remains only to check sums of two distinct elements. For $P_0$, $100+010=110\in P_1$.
For $P_1$, $001+110=111\in P_2$, 
$001+101=100\in P_0$ and 
$110+101=011\in P_2$.
For $P_2$, $011+111=100\in P_0$.
So, for each $i\in\{0,1,2\}$, no sum of two elements of $P_i$ lies in $P_i$.
\end{proof}

Fix a positive integer $k$. We now define a  graph $G_k$ as follows. Its vertex set is
\[
V(G_k)=Q\times\{0,1,2\}\times\{1,\ldots,k\}.
\]
For $x\in Q$, $i\in \{0,1,2\}$ and $r\in\{1,\ldots,k\}$, 
we write the vertex $(x,i,r)$ as $x_i^{(r)}$. The index $i\in\{0,1,2\}$ is called the \emph{port} of  $x_i^{(r)}$. Obviously, 
$|V(G_k)|=8\cdot 3\cdot k=24k$.
The edges of $G_k$ are of two types.

\begin{enumerate}[label=(\roman*)]
\item Internal edges. For every $x\in Q$ and every $r\in\{1,\ldots,k\}$,  $x_0^{(r)}, x_1^{(r)},  x_2^{(r)}$ are pairwise adjacent, so they form  a \emph{small triangle}, written as $T_x^{(r)}$ when $x$ and $r$ are fixed. 

\item External edges. If $x,y\in Q$ are distinct, $r,s\in\{1,\ldots,k\}$, and $x+y\in P_i$, then $x_i^{(r)}$ and $y_i^{(s)}$ (with the same port $i$)  are adjacent. We call the edge $x_i^{(r)}y_i^{(s)}$ an $i$-external edge.
%
\end{enumerate}

The structure of $G_k$ is simple: internal edges stay inside each small triangle $T_x^{(r)}$; external edges connect vertices with different $Q$-labels, but only when they share the same port $i$ and their $Q$-labels sum to an element of $P_i$. 
To better understand the external edges, we first fix a port $i \in \{0,1,2\}$ and illustrate how vertices from different $Q$-labels are connected via this port. This is shown in Figure~\ref{fig:schematic}. For example with $x=100\in Q$ and $y=000\in Q$, as $x+y=100\in P_0$, the line between $x=100$ and $y=000$ stands for the edges between $x_0^{(r)}$ and $y_0^{(s)}$ for all $r,s\in \{1,\dots,k\}$.
With the port structure in mind, we now present the full graph $G_1$ in Figure~\ref{fig:G1}.


\begin{figure}[!htbp]
\centering
\resizebox{0.95\textwidth}{!}{%
\begin{tikzpicture}[
    q0/.style={v0,minimum size=15.5pt,inner sep=1.2pt,font=\tiny,text=black},
    q1/.style={v1,minimum size=15.5pt,inner sep=1.2pt,font=\tiny,text=black},
    q2/.style={v2,minimum size=15.5pt,inner sep=1.2pt,font=\tiny,text=black},
    paneltitle/.style={font=\small\bfseries,align=center,text width=3.2cm},
    panelnote/.style={font=\scriptsize,align=center,text width=3.6cm}
]
\def\labellist{0/000/90,1/001/45,3/011/0,2/010/315,6/110/270,7/111/225,5/101/180,4/100/135}

\begin{scope}[shift={(0,0)}]
  \foreach \id/\lab/\ang in \labellist{
    \coordinate (g0c-\id) at (\ang:1.42);
  }

  \foreach \a/\b in {0/4,1/5,2/6,3/7,0/2,1/3,4/6,5/7}{
    \draw[p0edge] (g0c-\a)--(g0c-\b);
  }

  \foreach \id/\lab/\ang in \labellist{
    \node[q0] at (g0c-\id) {\lab};
  }

  \node[panelnote] at (0,-2.2) {port $0$\\$P_0=\{100,010\}$};
\end{scope}

\begin{scope}[shift={(4.15,0)}]
  \foreach \id/\lab/\ang in \labellist{
    \coordinate (g1c-\id) at (\ang:1.42);
  }

  \foreach \a/\b in {0/1,2/3,4/5,6/7,0/6,1/7,2/4,3/5,0/5,1/4,2/7,3/6}{
    \draw[p1edge] (g1c-\a)--(g1c-\b);
  }

  \foreach \id/\lab/\ang in \labellist{
    \node[q1] at (g1c-\id) {\lab};
  }

  \node[panelnote] at (0,-2.2) {port $1$\\$P_1=\{001,110,101\}$};
\end{scope}

\begin{scope}[shift={(8.30,0)}]
  \foreach \id/\lab/\ang in \labellist{
    \coordinate (g2c-\id) at (\ang:1.42);
  }

  \foreach \a/\b in {0/3,1/2,4/7,5/6,0/7,1/6,2/5,3/4}{
    \draw[p2edge] (g2c-\a)--(g2c-\b);
  }

  \foreach \id/\lab/\ang in \labellist{
    \node[q2] at (g2c-\id) {\lab};
  }

  \node[panelnote] at (0,-2.2){port $2$\\$P_2=\{011,111\}$};
\end{scope}
\end{tikzpicture}%
}
\caption{Illustration of the external edges through a fixed port $i\in \{0,1,2\}$.}
\label{fig:schematic}
\end{figure}
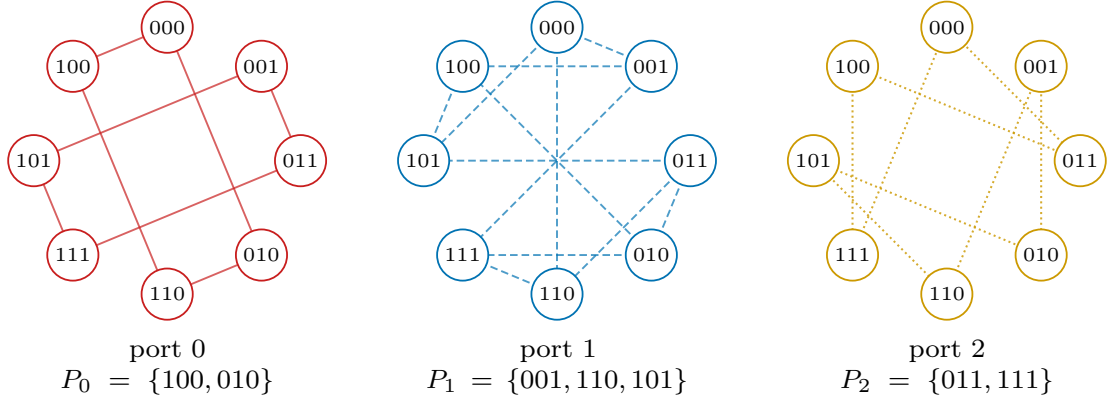

\section{Proof of Theorem~\ref{thm:main}}\label{sec:proof-main}

\begin{figure}[!htbp]
\centering
\begin{tikzpicture}[x=1cm,y=1cm]
\coordinate (A000p0) at (-0.800,0.750);
\coordinate (A000p1) at (-1.460,-0.413);
\coordinate (A000p2) at (-0.140,-0.413);
\coordinate (A100p0) at (3.050,0.750);
\coordinate (A100p1) at (2.390,-0.413);
\coordinate (A100p2) at (3.710,-0.413);
\coordinate (A010p0) at (-0.800,4.750);
\coordinate (A010p1) at (-1.460,3.587);
\coordinate (A010p2) at (-0.140,3.587);
\coordinate (A110p0) at (3.050,4.750);
\coordinate (A110p1) at (2.390,3.587);
\coordinate (A110p2) at (3.710,3.587);
\coordinate (A001p0) at (1.800,2.750);
\coordinate (A001p1) at (1.140,1.587);
\coordinate (A001p2) at (2.460,1.587);
\coordinate (A101p0) at (5.650,2.750);
\coordinate (A101p1) at (4.990,1.587);
\coordinate (A101p2) at (6.310,1.587);
\coordinate (A011p0) at (1.800,6.750);
\coordinate (A011p1) at (1.140,5.587);
\coordinate (A011p2) at (2.460,5.587);
\coordinate (A111p0) at (5.650,6.750);
\coordinate (A111p1) at (4.990,5.587);
\coordinate (A111p2) at (6.310,5.587);
\draw[p0edge] (A000p0) -- (A100p0);
\draw[p0edge] (A000p0) -- (A010p0);
\draw[p0edge] (A100p0) -- (A110p0);
\draw[p0edge] (A010p0) -- (A110p0);
\draw[p0edge] (A001p0) -- (A101p0);
\draw[p0edge] (A001p0) -- (A011p0);
\draw[p0edge] (A101p0) -- (A111p0);
\draw[p0edge] (A011p0) -- (A111p0);
\draw[p1edge] (A000p1) -- (A110p1);
\draw[p1edge] (A000p1) -- (A001p1);
\draw[p1edge] (A000p1) -- (A101p1);
\draw[p1edge] (A100p1) -- (A010p1);
\draw[p1edge] (A100p1) -- (A001p1);
\draw[p1edge] (A100p1) -- (A101p1);
\draw[p1edge] (A010p1) -- (A011p1);
\draw[p1edge] (A010p1) -- (A111p1);
\draw[p1edge] (A110p1) -- (A011p1);
\draw[p1edge] (A110p1) -- (A111p1);
\draw[p1edge] (A001p1) -- (A111p1);
\draw[p1edge] (A101p1) -- (A011p1);
\draw[p2edge] (A000p2) -- (A011p2);
\draw[p2edge] (A000p2) -- (A111p2);
\draw[p2edge] (A100p2) -- (A011p2);
\draw[p2edge] (A100p2) -- (A111p2);
\draw[p2edge] (A010p2) -- (A001p2);
\draw[p2edge] (A010p2) -- (A101p2);
\draw[p2edge] (A110p2) -- (A001p2);
\draw[p2edge] (A110p2) -- (A101p2);
\draw[internal] (A000p0) -- (A000p1) -- (A000p2) -- cycle;
\node[v0] at (A000p0) {};
\node[v1] at (A000p1) {};
\node[v2] at (A000p2) {};
\node[labelnode] at (-0.800,0.000) {$\mathtt{000}$};
\draw[internal] (A100p0) -- (A100p1) -- (A100p2) -- cycle;
\node[v0] at (A100p0) {};
\node[v1] at (A100p1) {};
\node[v2] at (A100p2) {};
\node[labelnode] at (3.050,0.000) {$\mathtt{100}$};
\draw[internal] (A010p0) -- (A010p1) -- (A010p2) -- cycle;
\node[v0] at (A010p0) {};
\node[v1] at (A010p1) {};
\node[v2] at (A010p2) {};
\node[labelnode] at (-0.800,4.000) {$\mathtt{010}$};
\draw[internal] (A110p0) -- (A110p1) -- (A110p2) -- cycle;
\node[v0] at (A110p0) {};
\node[v1] at (A110p1) {};
\node[v2] at (A110p2) {};
\node[labelnode] at (3.050,4.000) {$\mathtt{110}$};
\draw[internal] (A001p0) -- (A001p1) -- (A001p2) -- cycle;
\node[v0] at (A001p0) {};
\node[v1] at (A001p1) {};
\node[v2] at (A001p2) {};
\node[labelnode] at (1.800,2.000) {$\mathtt{001}$};
\draw[internal] (A101p0) -- (A101p1) -- (A101p2) -- cycle;
\node[v0] at (A101p0) {};
\node[v1] at (A101p1) {};
\node[v2] at (A101p2) {};
\node[labelnode] at (5.650,2.000) {$\mathtt{101}$};
\draw[internal] (A011p0) -- (A011p1) -- (A011p2) -- cycle;
\node[v0] at (A011p0) {};
\node[v1] at (A011p1) {};
\node[v2] at (A011p2) {};
\node[labelnode] at (1.800,6.000) {$\mathtt{011}$};
\draw[internal] (A111p0) -- (A111p1) -- (A111p2) -- cycle;
\node[v0] at (A111p0) {};
\node[v1] at (A111p1) {};
\node[v2] at (A111p2) {};
\node[labelnode] at (5.650,6.000) {$\mathtt{111}$};
\begin{scope}[shift={(8.7,5.85)},scale=0.70]
  \node[font=\scriptsize] at (2.5,0.2) {ports in each $T_x^{(1)}$};
  \coordinate (Lg0) at (0,0.62); \coordinate (Lg1) at (-0.58,-0.42); \coordinate (Lg2) at (0.58,-0.42);
  \draw[internal] (Lg0)--(Lg1)--(Lg2)--cycle;
  \node[v0] at (Lg0) {}; \node[v1] at (Lg1) {}; \node[v2] at (Lg2) {};
  \node[font=\scriptsize] at (0,0.94) {$0$};
  \node[font=\scriptsize] at (-0.87,-0.55) {$1$};
  \node[font=\scriptsize] at (0.87,-0.55) {$2$};
  \draw[internal] (-0.3,-1.2) -- (0.5,-1.2);
  \node[right,black,font=\scriptsize] at (0.5,-1.2) {internal edges};
  \draw[p0edge, legendline] (-0.3,-2.2) -- (0.5,-2.2);
  \node[right, black, font=\scriptsize] at (0.5,-2.2) {$0$-external edges};
  \draw[p1edge, legendline] (-0.3,-3.2) -- (0.5,-3.2);
  \node[right, black, font=\scriptsize] at (0.5,-3.2) {$1$-external edges};
  \draw[p2edge, legendline] (-0.3,-4.2) -- (0.5,-4.2);
  \node[right, black, font=\scriptsize] at (0.5,-4.2) {$2$-external edges};
\end{scope}
\end{tikzpicture}
\caption{The graph $G_1$.}
\label{fig:G1}
\end{figure}
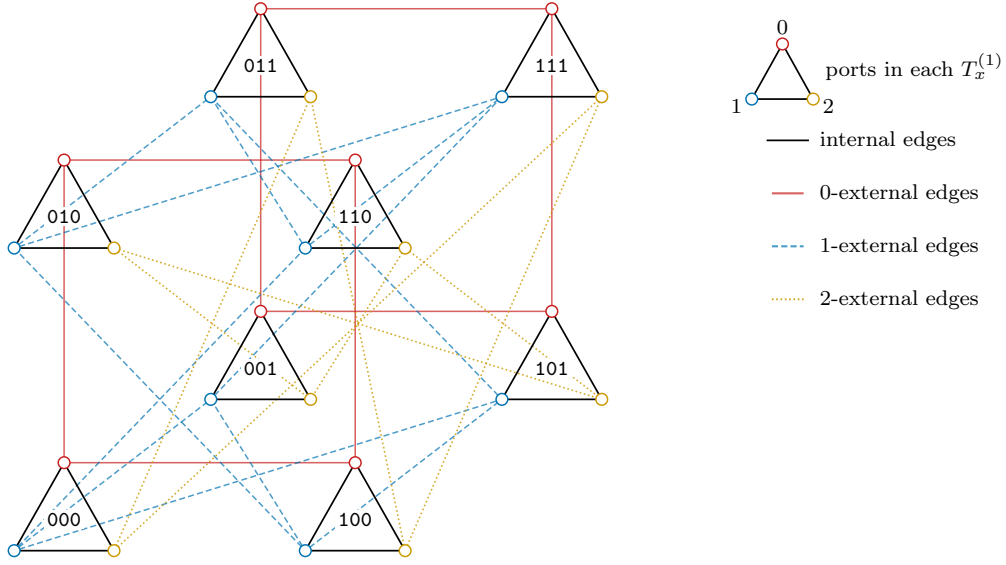

We prove Theorem~\ref{thm:main} in this section by showing that $G_k$ is such a graph. 
The connectivity bound and the absence of pancyclic edges are established in Subsections~\ref{sec:connectivity} and \ref{sec:noedge}, respectively, while Subsection~\ref{sec:vertex} proves vertex-pancyclicity via direct construction of short cycles and a properly colored auxiliary graph for longer cycles.

Consider the graph  $G_k$.
For $r\in\{1,\ldots,k\}$, we call the set
\[
L_r=\{x_i^{(r)}:x\in Q,\, i\in\{0,1,2\}\}
\]
the $r$-th layer.

\subsection{Connectivity}\label{sec:connectivity}

\begin{lemma}\label{lem:connectivity}
For every positive integer $k$, the graph $G_k$ is $k$-connected.
\end{lemma}

\begin{proof}
It is enough to prove that deleting fewer than $k$ vertices leaves a connected graph.

We first claim that each layer $G_k[L_r]$ is connected. Indeed, if $x,y\in Q$ and $x\ne y$, then $x+y\in P_i$ for some unique $i\in\{0,1,2\}$, and so $x_i^{(r)}y_i^{(r)}$ is an edge. Since each $T_x^{(r)}$ is a triangle, it follows that all small triangles in $L_r$ are connected to each other, and hence $G_k[L_r]$ is connected.

Let $G_k'$ be a graph obtained from $G_k$ by deleting fewer than $k$ vertices. Since there are $k$ layers, at least one layer, say $L_{r_0}$ is unchanged,  Then  $G_k'[L_{r_0}]\cong G_k[L_{r_0}]$ is connected.

Let $x_i^{(r)}$ be a vertex not in $L_{r_0}$. As $P_i$ is nonempty, choose $a\in P_i$. Then $x+a\ne x$. Recall that $(x+a)_i^{(r_0)}\in V(G_k')$. By the definition of external edges, we have $x_i^{(r)}(x+a)_i^{(r_0)}\in E(G_k')$.
Therefore every vertex not in $L_{r_0}$ is adjacent to the connected subgraph $G_t'[L_{r_0}]$, which implies that $G_k'$ is connected. 
\end{proof}

In fact, we may show that the connectivity of $G_k$ is $2k+2$

\begin{lemma}\label{layer}
For a layer $L_r$, $G[L_r]$ is $4$-connected.
\end{lemma}
\begin{proof}
For $x\in Q$, we write $T_x$ for $T_x^{(r)}$ in the layer $L_r$ (i.e., omit the layer number). Let $S$ be any vertex set of $L_r$ with $|S|\le 3$. It suffices to show that $G[L_r]-S$ is connected. 

We call a triangle $T_x$ is intact if no vertex is deleted (i.e., $T_x\cap S=\emptyset$), and damaged otherwise. We say two triangles $T_x$ and $T_y$ are adjacent if there is an edge between some vertex of $T_x$ and some of $T_y$. As $|S|\le 3$, there are at most three damaged triangles.

For $x\ne y$, if $T_x$ and $T_y$ are intact, as $x+y\in P_i$ for some $i\in \{0,1,2\}$, then $T_x$ and $T_y$ are adjacent by $x_iy_i$. 

It remains to show that any damaged triangle is adjacent to some intact triangle. Suppose to the contrary that some damaged triangle $T_x$ is not adjacent to any intact triangle. Let $I=\{i=0,1,2:x_i\notin S\}$. Then $|I|=1,2$. Let $i\in I$. For any $a\in P_i$, $x_i$ is adjacent to $(x+a)_i$, and so $T_{x+a}$ is damaged. As $|P_i|\ge 2$, together with $T_x$, there are at least three damaged triangles. However, $|S|\le 3$. There are exactly three damaged triangles, implying that $|I|=2$. Let $j\in I\setminus \{i\}$. For any $b\in P_j$, $x_j$ is adjacent to $(x+b)_j$, and so $T_{x+b}$ is damaged. Note that for any $a\in P_i$ and $b\in P_j$, as $x+a\ne x+b$, $T_{x+a}$ is different from $T_{x+b}$, implying that there are at least four damaged triangles, a contradiction. 
So any damaged triangle is connected to some intact triangle, proving the lemma.
\end{proof}

\begin{proposition} $\kappa(G)= 
2k+2$.
\end{proposition}
\begin{proof}
For $i\in \{0,2\}$, $d(x_i^{(r)})=2+k|P_i|=2+2k$, we have by Whitney's theorem that $\kappa(G_k)\le 2+2k$.
Let $S$ be a vertex set with $|S|\le 1+2k$. It suffices to show that $G_k-S$ is connected. If $k=1$, then $G_1$ has one layer and so the result follows by Lemma \ref{layer}. Suppose in the following that $k\ge 2$.

As there are $k$ layers and $|S|\le 1+2k$, there exists one layer $L:=L_{r_0}$ such that $|S\cap L|\le 2$. By Lemma \ref{layer}, $G_k[L\setminus S]$ is connected. 
It suffices to show that any vertex in $V(G_k)\setminus (L\cup S)$ is connected to $L\setminus S$.

Suppose to the contrary that there is a vertex $u:=x_i^{(r)}\in V(G_k)\setminus (L\cup S)$ that is not connected to $L\setminus S$. 
Note that $N_L(u)=\{(x+a)_i^{(r_0)}:a\in P_i\}$. If $N_L(u)\setminus S\ne \emptyset$, then $u$ is adjacent to some vertex of $L\setminus S$, a contradiction. Then $N_L(u)\subseteq S$. As $|S\cap L|\le 2$, we have $i\in \{0,2\}$ and $S\cap L=N_L(u)$. 
We will show that $N(u)\subseteq S$. 
If $x_j^{(r)}\notin S$ with $j\ne i$, then for any $a\in P_j$, $(x+a)_j^{(r_0)}\in L\setminus S$ is adjacent to $x_j^r$, implying that $ux_j^{(r)}(x+a)_j^{(r_0)}$ is a path connecting $u$ and $L\setminus S$, a contradiction. This shows that $x_j^r\in S$ for any $j\ne i$. If $(x+a)_i^{(s)}\notin S$ with $a\in P_i$ and $s\in \{1,\dots,k\}$, then $s\ne r_0$ and $(x+a)_i^{(s)}$ is adjacent to $x_i^{(r_0)}\in L$. Note that $x_i^{(r_0)}\notin S$. We have $u(x+a)_i^{(s)}x_i^{(r_0)}$ is a path connecting $u$ and $L\setminus S$, also a contradiction. This shows that  $(x+a)_i^{(s)}\in S$ for any $a\in P_i$ and $s\in \{1,\dots,k\}$. So $N(u)\subseteq S$. 
It then follows that $2+2k=2+k|P_i|\le d(u)\le |S|$, a contradiction. 

Therefore, any vertex in $V(G_k)\setminus (L\cup S)$ is connected to $L\setminus S$, implying that $G_k-S$ is connected. So  $\kappa(G_k)= 2+2k$.
\end{proof}
\subsection{Absence of pancyclic edges}\label{sec:noedge}

The next lemma is the main local reason why no edge can be pancyclic.

\begin{lemma}\label{lem:no-pancyclic-edge}
The graph $G_k$ contains no pancyclic edge.
\end{lemma}

\begin{proof}
We show that every edge misses at least one of the two shortest possible cycle lengths. More precisely, every external edge is contained in no triangle, and every internal edge is contained in no $4$-cycle.

First let $e=x_i^{(r)}y_i^{(s)}$ 
be an $i$-external edge for some $i=0,1,2$. Then $x\ne y$ and $x+y\in P_i$. Suppose to the contrary that $e$ lies in a triangle. Then $x_i^{(r)}$ and $y_i^{(s)}$ have a common neighbor, say $z_j^{(q)}$.
If $z_j^{(q)}x_i^{(r)}$ is an internal edge, then $z=x$, $q=r$, and $j\ne i$. However, $z_j^{(q)}$ is not adjacent to $y_i^{(s)}$, a contradiction. The same argument applies with $x_i^{(r)}$ and $y_i^{(s)}$ interchanged. 
This implies that $z_j^{(q)}x_i^{(r)}$ and  $z_j^{(q)}y_i^{(s)}$ are external edges.
Therefore $j=i$ and
\[
x+z\in P_i,\qquad y+z\in P_i.
\]
Then
\[
(x+z)+(y+z)=x+y\in P_i.
\]
This contradicts the sum-freeness of $P_i$ from Lemma~\ref{lem:sumfree} 
as $x+z$ and $y+z$ are distinct. Thus no external edge lies in a triangle.

Now let $e=x_i^{(r)}x_j^{(r)}$
be an internal edge. Then $i\ne j$. Suppose to the contrary that $e$ lies in a $4$-cycle. Then, after deleting the edge $e$, there is a path of the form  $x_i^{(r)}uvx_j^{(r)}$ with length three.  
If $x_i^{(r)}u$ is an internal edge, then $u=x_h^{(r)}$, where $h$ is the unique element of $\{0,1,2\}\setminus\{i,j\}$. Consider the next edge $uv$. If $uv$ is internal, then $v$ would have to be either $x_i^{(r)}$ or $x_j^{(r)}$, a contradiction. So $uv$ is an external edge, implying that $v=y_h^{(p)}$ for some $y\ne x$. As $h\ne j$, $vx_j^{(r)}$ cannot be an external edge,  so it is an internal edge. However, $y\ne x$, $vx_j^{(r)}$ cannot be an internal edge, a contradiction. Therefore $x_i^{(r)}u$ is an external edge. By the same argument, $vx_j^{(r)}$ is external.
It follows that
\[
u=y_i^{(p)},\qquad v=z_j^{(q)}
\]
for some $y,z\in Q$ and some $p,q\in \{1,\dots,k\}$ with
\[
x+y\in P_i,
\qquad
x+z\in P_j.
\]
As $i\ne j$, $uv$ cannot be an external edge. Thus $uv$ must be internal, so $y=z$ and $p=q$. Consequently
\[
x+y\in P_i\cap P_j=\emptyset,
\]
which is impossible. Hence no internal edge lies in a $4$-cycle.

Therefore every edge of $G_k$ fails to lie in at least one cycle length $3$ or $4$. 
\end{proof}

\subsection{Vertex-pancyclicity}\label{sec:vertex}

It remains to prove that every vertex lies in cycles of all lengths from $3$ to $24k$. We first dispose of the short lengths.

\begin{lemma}\label{lem:short-cycles}
Every vertex of $G_k$ lies in a cycle of each length $3,4,5,6$.
\end{lemma}

\begin{proof}
Let $x_i^{(r)}$ be a vertex of $G_k$. By the definition of internal edges, the small triangle $T_x^{(r)}$ gives a $3$-cycle through $x_i^{(r)}$.

For a $4$-cycle, choose two distinct elements $a,b\in P_i$. Then
\[
x_i^{(r)},\quad (x+a)_i^{(r)},\quad (x+a+b)_i^{(r)},\quad (x+b)_i^{(r)}
\]
are four distinct vertices. The consecutive pairs in the displayed order have label differences $a, b, a, b$, respectively, all of which lie in $P_i$. Hence these four vertices give a $4$-cycle through $x_i^{(r)}$.

For a $5$-cycle, again choose distinct $a,b\in P_i$. By Lemma~\ref{lem:sumfree}, $a+b\notin P_i$. Let $j$ be the unique index with $a+b\in P_j$. Then $j\ne i$ and
\[
x_i^{(r)} (x+a)_i^{(r)}
(x+a)_j^{(r)}
(x+b)_j^{(r)}
(x+b)_i^{(r)}
x_i^{(r)}
\]
is a $5$-cycle. 

For a $6$-cycle, we use the following choices. If $i=0$, take
\[
a=010\in P_0,
\qquad b=001\in P_1,
\]
so $a+b=011\in P_2$. If $i=1$, take
\[
a=001\in P_1,
\qquad b=010\in P_0,
\]
so $a+b=011\in P_2$. If $i=2$, take
\[
a=011\in P_2,
\qquad b=001\in P_1,
\]
so $a+b=010\in P_0$.
Thus in every case we have chosen $a\in P_i$, $b\in P_j$, and $a+b\in P_k$, where $i,j,k$ are pairwise distinct. Then
\[
x_i^{(r)}
x_j^{(r)}
(x+b)_j^{(r)}
(x+b)_k^{(r)}
(x+a)_k^{(r)}
(x+a)_i^{(r)}
x_i^{(r)}
\]
is a $6$-cycle through $x_i^{(r)}$. 
\end{proof}

We now construct longer cycles by using an auxiliary edge-colored graph.

Let $R_k$ be the edge-colored graph with vertex set
\[
\{x^{(r)}:x\in Q,\; r\in\{1,\ldots,k\}\},
\]
edge set 
\[
\{x^{(r)}y^{(s)}: x,y\in Q, r,s\in \{1,\dots, k\}, x\ne y\},
\]
 and the color of  edge  $x^{(r)}y^{(s)}$ is $i$, which is the unique $i=0,1,2$ with $x+y\in P_i$. 

A cycle (path, respectively) in $R_k$ is called \emph{properly colored} if any two adjacent edges on the cycle (path, respectively) have different colors. A path on $s$ vertices is denoted by $L_s$.

\begin{lemma}\label{lem:proper-cycles}
For every $q$ with $3\le q\le 8k$ and every vertex $x^{(r)}$ of $R_k$, there exists a properly colored $q$-cycle in $R_k$ containing $x^{(r)}$.
\end{lemma}

\begin{proof}
Let
\[
a=100,
\qquad
b=010,
\qquad
c=001.
\]
We first list six paths in any layer (we omit the layer number here), all starting at $z=000$ in the following table,  the color of an edge between two labels $x,y\in Q$ is the unique index $i$ such that $x+y\in P_i$. 
\[
\begin{array}{c|l|l}
\ell & \text{path }L_\ell & \text{color sequence along }L_\ell\\
\hline
3 & zc(b+c) & 1,0 \\
4 & zcb(b+c) & 1,2,1 \\
5 & zcb(b+c)(a+b+c) & 1,2,1,0\\
6 & zcba(b+c)(a+b+c) & 1,2,1,2,0 \\
7 & zcb(b+c)a(a+c)(a+b+c) & 1,2,1,2,1,0 \\
8 & zcb(b+c)a(a+c)(a+b)(a+b+c) & 1,2,1,2,1,2,1
\end{array}
\]
 In particular, for every $\ell\in\{3,4,5,6,7,8\}$:
\begin{enumerate}[label=(\alph*)]
\item $L_\ell$ is a properly colored path on $\ell$ vertices;
\item the first edge of $L_\ell$ has color $1$;
\item the last vertex of $L_\ell$ belongs to $P_2$;
\item the color of the last edge of $L_\ell$ is different from $2$.
\end{enumerate}

Now fix $q$ with $3\le q\le 8k$ and fix a vertex $x^{(r)}$ of $R_k$. Let $s=\left\lceil \frac{q}{8}\right\rceil$.
Then $s\le k$ and
\[
3s\le q\le 8s.
\]
Therefore we may write
$q=q_1+q_2+\cdots+q_s$
with each $q_j\in\{3,\dots,8\}$. 

Choose distinct layers $r_1,\ldots,r_s$ with $r_1=r$. In layer $r_j$, place a `translated' copy of the path $L_{q_j}$, obtained by adding $x$ to every label. We denote the resulting path by $L_j^*$. Thus $L_j^*$ starts at $x^{(r_j)}$ and ends at $(x+h_j)^{(r_j)}$ for some $h_j\in P_2$.

For $j=1,\ldots,s$, note that $(x+h_j)^{(r_j)}x^{(r_{j+1})}\in E(R_k)$, where the index $j$ takes modulo $s$, and as $h_j\in P_2$, the edge has color $2$.
Let \[
C=x^{r}L_1^*(x+h_1)^{(r_1)}x^{(r_{2})}L_2^*\cdots L_s^*(x+h_s)^{(r_s)}x^{(r)}.
\]
Then $C$ is a cycle of length $q_1+q_2+\cdots+q_s=q$. 
For $j=1,\dots,s$, each $L_j^*$ is properly colored, and neither the first edge nor the last edge has color $2$. Hence $C$ is properly colored, as desired.
\end{proof}

\begin{lemma}\label{lem:lift}
Let $C$ be a properly colored $q$-cycle in $R_k$ containing $x^{(r)}$. For every integer $m$ with $2q< m\le 3q$, and for every port $i\in\{0,1,2\}$, the graph $G_k$ contains an $m$-cycle through $x_i^{(r)}$.
\end{lemma}

\begin{proof}
Let $C=X_1\dots X_qX_1$ with  $X_1=x^{(r)}$. Let $X_j=\xi_j^{(\rho_j)}$ with $\xi_j\in Q$ and $\rho_j\in\{1,\ldots,k\}$. Let $c_j$ be the color of the edge $X_jX_{j+1}$ with indices modulo $q$. Then $c_0=c_q$.

We lift a properly colored $q$-cycle $C = X_1\ldots X_qX_1$ in $R_k$ to a cycle in $G_k$. 
For each $X_j = \xi_j^{(\rho_j)}$ in $C$, we take the small triangle $T_{\xi_j}^{(\rho_j)}$ in $G_k$.
The edge $X_{j-1}X_j$ has color $c_{j-1}$, which corresponds to the $c_{j-1}$-extremal edge between $T_{\xi_{j-1}}^{(\rho_{j-1})}$ and $T_{\xi_j}^{(\rho_j)}$.
Since $C$ is properly colored, $c_{j-1}\ne c_j$ for every $j$. To obtain a cycle in $G_k$ from $R_k$ using $C$, in the small triangle $T_{\xi_j}^{(\rho_j)}$ of $G_k$, we must move from port $c_{j-1}$ to port $c_j$. There are two possible internal paths:
\begin{itemize}
\item the short path $(\xi_j)_{c_{j-1}}^{(\rho_j)} (\xi_j)_{c_j}^{(\rho_j)}$;
\item the long path $(\xi_j)_{c_{j-1}}^{(\rho_j)} (\xi_j)_d^{(\rho_j)} (\xi_j)_{c_j}^{(\rho_j)}$,
where $d$ is the unique element of $\{0,1,2\}\setminus\{c_{j-1},c_j\}$.
\end{itemize}
After the chosen internal path in $T_{\xi_j}^{(\rho_j)}$, we use the $c_j$-external edge of color $c_j$ to move to the next small triangle. This edge exists because $X_jX_{j+1}$ has color $c_j$ in $R_k$.

Thus the properly colored $q$-cycle in $R_k$ lifts to a simple cycle in $G_k$. The vertices $X_1,\ldots,X_q$ are distinct in $R_t$, and hence each corresponding small triangle of $G_t$ is used at most once. If all $q$ small triangles use the short path, the lifted cycle has length $2q$. Each time we replace one short path by the long path, the length increases by $1$. Hence we can obtain lifted cycles of all lengths between $2q$ and $3q$. 

If $i\in \{c_0,c_1\}$, then even using a short path at $X_1$ passes through $x_i^{(r)}$, so any $m$-cycle ($2q\le m\le 3q$) through $x_i^{(r)}$ is achievable. If $i \notin \{c_0, c_1\}$, then a long path must be used at $X_1$, this requires $m \ge 2q + 1$, and any $m$-cycle ($2q+1\le m\le 3q$) through $x_i^{(r)}$ is achievable. This completes the proof.
\end{proof}

\begin{proposition}\label{prop:vertex-pancyclic}
For every positive integer $k$, the graph $G_k$ is vertex-pancyclic.
\end{proposition}

\begin{proof}
Let $x_i^{(r)}$ be an arbitrary vertex of $G_k$. By Lemma~\ref{lem:short-cycles}, $x_i^{(r)}$ lies in cycles of lengths $3,4,5,6$.
Now let $m$ be an integer with $7\le m\le 24k$.
Let $q=\left\lceil \frac{m}{3}\right\rceil$.
Then $3\le q\le 8k$ and $2q<m\le 3q$.
By Lemma~\ref{lem:proper-cycles}, there is a properly colored $q$-cycle in $R_t$ containing $x^{(r)}$. By Lemma~\ref{lem:lift}, this cycle lifts to an $m$-cycle of $G_t$ containing $x_i^{(r)}$. Therefore $x_i^{(r)}$ lies in a cycle of every length $3,\ldots,24k$.
\end{proof}

We are now ready to prove Theorem~\ref{thm:main}.

\begin{proof}[Proof of Theorem~\ref{thm:main}]
Given $k\ge 1$, consider the graph $G_k$ constructed in Section~\ref{sec:construction}. By Lemma~\ref{lem:connectivity}, $G_k$ is $k$-connected. By Proposition~\ref{prop:vertex-pancyclic}, $G_k$ is vertex-pancyclic. By Lemma~\ref{lem:no-pancyclic-edge}, $G_k$ contains no pancyclic edge, as desired.
\end{proof}



\begin{thebibliography}{99}

\bibitem{Bondy1971}
J.A. Bondy,
Pancyclic graphs. I,
\emph{J. Combin. Theory Ser. B} \textbf{11} (1971) 80--84.

\bibitem{BondyIngleton1976}
J.A. Bondy, A. W. Ingleton,
Pancyclic graphs. II,
\emph{J. Combin. Theory Ser. B} \textbf{20} (1976) 41--46.

\bibitem{Broersma1997}
H.J. Broersma,
A note on the minimum size of a vertex pancyclic graph,
\emph{Discrete Math.} \textbf{164} (1997) 29--32.

\bibitem{ChvatalErdos1972}
V. Chv\'{a}tal,  P. Erd\H{o}s,
A note on Hamiltonian circuits,
\emph{Discrete Math.} \textbf{2} (1972) 111--113.

\bibitem{CreamGouldHirohata2018}
M. Cream, R.J. Gould, K. Hirohata,
Extending vertex and edge pancyclic graphs,
\emph{Graphs Combin.} \textbf{34} (2018) 1691--1711.

\bibitem{Gould2003}
R.J. Gould,
Advances on the Hamiltonian problem: a survey,
\emph{Graphs Combin.} \textbf{19} (2003) 7--52.

\bibitem{Hendry1990}
G.R.T. Hendry,
Extending cycles in graphs,
\emph{Discrete Math.} \textbf{85} (1990) 59--72.

\bibitem{Hobbs1976}
A.M. Hobbs,
The square of a block is vertex pancyclic,
\emph{J. Combin. Theory Ser. B} \textbf{20} (1976) 1--4.

\bibitem{LiZhan2026}
C. Li, X. Zhan,
Every $2$-connected $[4,2]$-graph of order at least seven contains a pancyclic edge,
\emph{Discrete Math.} \textbf{349} (2026) 115154.

\bibitem{Randerath2002}
B. Randerath, I. Schiermeyer, M. Tewes, L. Volkmann,
Vertex pancyclic graphs,
\emph{Discrete Appl. Math.} \textbf{120} (2002) 219--237.

\bibitem{SchmeichelHakimi1974}
E.F. Schmeichel, S.L. Hakimi,
Pancyclic graphs and a conjecture of Bondy and Chv\'{a}tal,
\emph{J. Combin. Theory Ser. B} \textbf{17} (1974) 22--34.

\bibitem{ZhangHolton1993}
K.M. Zhang, D.A. Holton,
On edge-pancyclic graphs,
\emph{Soochow J. Math.} \textbf{19} (1993) 37--41.

\end{thebibliography}
\end{document}